\documentclass[11pt]{article}
\usepackage{amsfonts,amsmath,amssymb}
\usepackage{hyperref,amsthm}

\newcommand{\handout}[5]{
   \renewcommand{\thepage}{#1-\arabic{page}}
   \noindent
   \begin{center}
   \framebox{
      \vbox{
    \hbox to 5.78in { {\bf (80630) Analytical Methods in Combinatorics and CS}
         \hfill #2 }
       \vspace{4mm}
       \hbox to 5.78in { {\Large \hfill #5  \hfill} }
       \vspace{2mm}
       \hbox to 5.78in { {\it #3 \hfill #4} }
      }
   }
   \end{center}
   \vspace*{4mm}
}

\newtheorem{theorem}{Theorem}

\newtheorem{lemma}[theorem]{Lemma}

\newtheorem{definition}[theorem]{Definition}
\newtheorem{claim}[theorem]{Claim}

\newcommand{\abs}[1]{\mathify{\left| #1 \right|}}

\newenvironment{proof-sketch}{\noindent{\bf Sketch of Proof}\hspace*{1em}}{\qed\bigskip}
\newenvironment{proof-idea}{\noindent{\bf Proof Idea}\hspace*{1em}}{\qed\bigskip}
\newenvironment{proof-of-lemma}[1]{\noindent{\bf Proof of Lemma #1}\hspace*{1em}}{\qed\bigskip}
\newenvironment{proof-attempt}{\noindent{\bf Proof Attempt}\hspace*{1em}}{\qed\bigskip}

\makeatletter
\def\fnum@figure{{\bf Figure \thefigure}}
\def\fnum@table{{\bf Table \thetable}}
\long\def\@mycaption#1[#2]#3{\addcontentsline{\csname
  ext@#1\endcsname}{#1}{\protect\numberline{\csname
  the#1\endcsname}{\ignorespaces #2}}\par
  \begingroup
    \@parboxrestore
    \small
    \@makecaption{\csname fnum@#1\endcsname}{\ignorespaces #3}\par
  \endgroup}
\def\mycaption{\refstepcounter\@captype \@dblarg{\@mycaption\@captype}}
\makeatother

\newcommand{\mathify}[1]{\ifmmode{#1}\else\mbox{$#1$}\fi}
\newcommand{\bigO}O
\newcommand{\set}[1]{\mathify{\left\{ #1 \right\}}}

\newcommand{\Var}{{\rm Var}}

\newcommand{\eqdef}{{\stackrel{\rm def}{=}}}

\providecommand{\norm}[1]{\left\| #1 \right\|}

\newcommand{\remove}[1]{}

\newcommand{\comment}[1]{}
\def\reals{\mathbb{R}}
\def\phi{\varphi}
\def\indicator{\mathbf{1}}

\newcommand\Parenth[1]{{\left( {#1} \right)}}

\newcommand\parenth[1]{{\bigl( {#1} \bigr)}}
\newcommand\brac[1]{{\left[ {#1} \right]}}
\newcommand\Brac[1]{{\bigl[ {#1} \bigr]}}

\newtheorem{notation}{Notation}

\renewenvironment{proof}[1][Proof]{
\noindent \textbf{#1: }}{$\Box$
\bigskip}

\newcommand{\ignore}[1]{}
\newcommand{\expect}{\mathbb{E}}

\renewcommand{\Pr}{{\bf Pr}}

\newcommand{\inner}[1]{\left\langle{#1}\right\rangle}

\newcommand{\stab}{\mathcal{S}}
\newcommand{\weight}{\mathcal{W}}

\newcommand{\nfrac}[2]{#1/#2}

\newcommand{\ch}{\omega}
\newcommand{\bool}{\set{0,1}}
\newcommand{\noise}{N}
\begin{document}

\title{Quantitative Relation Between Noise Sensitivity and Influences}

\author{
Nathan Keller\thanks{Faculty of Mathematics and Computer Science, The Weizmann Institute of Science, Rehovot, Israel.
Partially supported by the Adams Fellowship Program of the Israeli Academy of Sciences and Humanities and by
the Koshland Center for Basic Research. E-mail: nathan.keller@weizmann.ac.il} and Guy Kindler\thanks{Incumbent of the Harry and Abe Sherman Lectureship Chair at the Hebrew Univeristy of Jerusalem. Supported by the Israel Science Foundation and by the Binational Science Foundation.
E-mail:gkindler@cs.huji.ac.il} \\
}

\maketitle

\begin{abstract}
  A Boolean function $f:\{0,1\}^n \to \{0,1\}$ is said to be noise
  sensitive if inserting a small random error in its argument makes
  the value of the function almost unpredictable. Benjamini, Kalai and
  Schramm~\cite{BKS} showed that if the sum of squares of influences
  in $f$ is close to zero then $f$ must be noise sensitive. We show a
  quantitative version of this result which does not depend on $n$,
  and prove that it is tight for certain parameters\comment{perhaps
    change after doing the unbiased tightness}.  Our results hold also
  for a general product measure $\mu_p$ on the discrete cube, as long
  as $\log 1/p \ll \log n$.  We note that in \cite{BKS}, a
  quantitative relation between the sum of squares of the influences
  and the noise sensitivity was also shown, but only when the sum of
  squares is bounded by $n^{-c}$ for a constant $c$.

  Our results require a generalization of a lemma of Talagrand on the
  Fourier coefficients of monotone Boolean functions.  In order to
  achieve it, we present a considerably shorter proof of Talagrand's
  lemma, which easily generalizes in various directions, including
  non-monotone functions.
\end{abstract}


\section{Introduction}
\label{sec:Introduction}

The noise sensitivity of a function $f:\bool^n\to\bool$ is a
measure of how likely its value is to change, when evaluated on a
slightly perturbed input. Noise sensitivity became an important
concept in various areas of research in recent years, with
applications in percolation theory, complexity theory, and
learning theory (see e.g. \cite{BKS}, \cite{Hastad:01},
\cite{BsJaTa:99}). We work with a dual notion of noise
sensitivity, namely noise stability, defined as follows.

\begin{definition}
  For $x\in\bool^n$, the $\epsilon$-noise perturbation of $x$,
  denoted by
  $\noise_\epsilon(x)$, is a distribution obtained from $x$ by
  independently keeping each coordinate of $x$ unchanged with
  probability $1-\epsilon$, and replacing it by a random value with
  probability $\epsilon$. For this purpose we assume that a product
  distribution $\mu$ on the discrete cube $\bool^n$ is defined,
  however we leave it implicit in the notation.

\medskip

\noindent  The noise stability of $f$ is defined by
  \[
  \stab_{\epsilon}(f)\eqdef\mathrm{COV}_{x\sim\mu,\
    y\sim\noise_\epsilon(x)}\brac{f(x),f(y)}.
  \]
\end{definition}

\noindent Roughly saying, a function is noise-sensitive if its noise stability is close to
zero.

\medskip

Another concept that was intensively studied in recent decades is
that of the influences of coordinates on a function.

\begin{definition}
Let $f:\bool^n\to\bool$ be a Boolean function, and let
$i\in\set{1,\ldots,n}$. The influence of the $i$'th coordinate on $f$
is defined as
\[
I_i(f)\eqdef\Pr_{x\sim\mu}\brac{f(x)\neq f(x\oplus e_i)},
\]
where $x\oplus e_i$ is the vector obtained from $x$ by flipping the
$i$'th coordinate.
\end{definition}

Influences were studied in economics for decades, and first found
their way into computer science in \cite{BenorLinial:90}, in the
context of cryptography. The study of influences has numerous
applications in mathematical physics, economics, and various areas
of computer science, such as cryptography, hardness of
approximation, and computational lower-bounds (see e.g.
\cite{LiMaNi:93}, \cite{DinSaf2005}, \cite{mossel-2009}, or the
survey \cite{SafraKalaiSurvey}).

\paragraph{Relations between influences and noise sensitivity.} The
noise sensitivity of a function and its influences both measure how
likely it is to change its value when the input is slightly perturbed.
It makes sense to study the relations between these concepts. Perhaps
counterintuitively, it turns out that functions with very low
influences must be very sensitive to noise. This phenomenon was first
shown in a paper by Benjamini, Kalai, and Schramm~\cite{BKS}. They
proved the following theorem:
\begin{theorem}[\cite{BKS}]
  Let $\set{f_m:\bool^{n_m}\to\bool}_{m=1,2,\dots}$ be a sequence of Boolean functions,
  such that
\[
\sum_{i=1}^{n_m}I_i(f)^2\xrightarrow{m\rightarrow\infty}0.
\]
Then for any $\epsilon>0$, $\stab_{\epsilon}(f_m)\xrightarrow{m\to\infty}0$.
\end{theorem}
The BKS theorem was proved in \cite{BKS} only with respect to the
uniform measure on the Boolean cube, and the case of highly biased
product measures was left open. However, for some applications, such as
in the study of threshold phenomena, one is often interested in biased
measures. Moreover, the BKS theorem is qualitative, and does now show
a concrete relation between the influences of a function and its noise
stability (a quantitative relation was shown in \cite{BKS}, but only
for the case where for a function $f:\bool^n\to\bool$, the sum of
squares of the influences is inverse polynomially small in $n$).

\subsection{Our results}

In this paper we show a quantitative version of the BKS theorem. With
respect to the uniform measure, we prove that
\begin{theorem}\label{thm:main}
  There exists a constant $C>.234$ such that the following holds. Let
  $f:\bool^{n} \rightarrow \bool$, and denote
\[
\weight(f)=\frac14\cdot\sum_{i=1}^{n} I_i(f)^2.
\]
Then \[
\stab_\epsilon(f) \leq 20\cdot
\weight(f)^{C\cdot \epsilon}.
\]
\end{theorem}

The main technical tool used in the proof of Theorem~\ref{thm:main},
and also in the original qualitative result of \cite{BKS}, is a
generalization of a lemma by Talagrand~\cite{Talagrand:96}. Talagrand's
result considers the Fourier-Walsh expansion of a monotone Boolean function,
and bounds its weight on second-level Fourier coefficients in terms of
its weight on first level coefficients. This lemma is of independent
interest, and was used by Talagrand to estimate the correlation
between monotone families\cite{Talagrand:96}, and the size of the
boundary of subsets of the discrete cube \cite{Talagrand:97}. The
generalization gives a similar bound on the weight on $d$-level coefficients.

While in the paper of \cite{BKS} only a qualitative estimate was given in
the generalization of Talagrand's lemma, the main technical tool in
our proof is a quantitative version of it. To obtain it, we simplify
the proof of Talagrand's lemma in a way which makes its generalization
quite simple and straightforward. Our result for the uniform measure is the following:

\begin{lemma}
  For all $d \geq 2$, and for every function $f:\{0,1\}^n \rightarrow
  \bool$ such that
\begin{equation}
  \weight(f) \leq
\exp(-2(d-1))
\label{Eq:Restriction3}
\end{equation}
(where $\weight(f)$ is as in Theorem~\ref{thm:main}), we have
\begin{equation}
\sum_{|S|=d} \hat f(S)^2 \leq \frac{5e}d \cdot \Parenth{
      \frac{2e}{d-1}}^{d-1}\cdot\weight(f)\cdot \bigl(\log\Parenth{\nfrac{d}{\weight(f)}}\bigr)^{d-1}.
\end{equation}
\label{Lemma-Main}
\end{lemma}

Lemma~\ref{Lemma-Main}, as well as Theorem~\ref{thm:main}, hold also
for functions into the segment $[-1,1]$, if one appropriately extends
the definition of the influence. Specifically, one should define
$I_i(f)=\norm{f(x^{i\leftarrow0})-f(x^{i\leftarrow1})}_1$ where
$x^{i\leftarrow a}$ is the vector obtained from $x$ by inserting $a$
in the $i$'th coordinate. This holds also for the biased case,
discussed below. For simplicity, we assume that $f$ is Boolean in the
proofs.

\medskip

We note that Talagrand also proves a ``decoupled'' version of his
lemma. While we do not need a decoupled version for the proof of
Theorem~\ref{thm:main}, we prove one for the sake of completeness in
the Appendix.

\paragraph{Biased measure.}
In the study of threshold phenomena, and for other applications, often
one is interested in biased measures rather than the uniform measure
over the discrete cube. Once the proof of Talagrand's lemma is
simplified, it becomes easier to apply it also for biased
measures. Below are our analogous results with respect to the p-biased
measure on the discrete cube. The coefficients below are with respect
to the ``$p$-biased'' Fourier-Walsh expansion (see
Section~\ref{Sec:Preliminaries}).

\begin{lemma}\label{Lemma-biased-main}
Let $f:\bool^{n} \rightarrow \bool$, and denote
  \[
  \weight(f)=p(1-p)\cdot\sum_{i=1}^{n} I_i(f)^2.
  \]
For all $d \geq 2$, if
  \begin{equation}
    \weight(f) \leq
    \exp(-2(d-1)),
    \label{Eq:Restriction3}
  \end{equation}
then we have
\begin{equation}\label{eq:29}
\sum_{|S|=d} \hat f(S)^2 \leq \frac{5e}d \cdot \Parenth{\frac{2B(p)\cdot
      e}{d-1}}^{d-1}\cdot\weight(f)\cdot \Parenth{\log\Parenth{\frac{d}{\weight(f)}}}^{d-1},
\end{equation}
where $B(p)$  is the hypercontractivity constant defined in
Section~\ref{Sec:Preliminaries}.
\end{lemma}

We note that a bound slightly weaker than in
Lemma~\ref{Lemma-biased-main} can be obtained from
Lemma~\ref{Lemma-Main} by a general reduction technique, as was
observed in \cite{keller10}. However we prove
Lemma~\ref{Lemma-biased-main} directly, and Lemma~\ref{Lemma-Main}
follows as an immediate corollary.

\medskip\noindent
Using Lemma~\ref{Lemma-biased-main} we can prove an analogue of
Theorem~\ref{thm:main} for the case of biased measure.

\begin{theorem}\label{thm:biased-main}
  For all $d \geq 2$, and for every function $f:\{0,1\}^n \rightarrow
  \bool$ the following holds. Denoting
  $\weight(f)=p(1-p)\cdot\sum_{i=1}^{n} I_i(f)^2$,
we have
\begin{equation}\label{eq:290}
\stab_\epsilon(f) \leq (6e+1)
\weight(f)^{\alpha(\epsilon)\cdot \epsilon},
\end{equation}
where \[\alpha(\epsilon)=\frac{1}{\epsilon+\log(2B(p)e) +
  3\log\log(2B(p)e)},\]
and $B(p)$ is the hypercontractivity constant defined in
Section~\ref{Sec:Preliminaries}.
\end{theorem}

We note that for small $p$, $B(p)\approx \frac{1}{p\log(1/p)}$, and thus
Theorem~\ref{thm:biased-main} is useful only when $\log(1/p)$ is
assymptotically smaller than $\log(n)$. Indeed, when $p$ is inverse
polynomially small in $n$ the BKS theorem does not hold even
qualitatively -- there exist functions which have assymptotically small influences but
are noise stable. The graph property of containing a triangle with
respect to the critical probability $p$ is an
example of such a function.

\paragraph{Tightness.} Our main result (Theorem~\ref{thm:biased-main}) is tight for small $p$ up to a
constant factor in the exponent of $\weight(f)$, which tends to $1$
for small $\epsilon$, and for $p=1/2$ it is tight up to a constant
factor in the exponent. In Section~\ref{Sec:Tightness} we prove this,
and also discuss the tightness of Lemma~\ref{Lemma-biased-main},
showing that it is essentially tight.

\paragraph{Organization.} This paper is organized as follows: in
Section~\ref{Sec:Preliminaries} we recall the definitions of the
biased Fourier-Walsh expansion, hypercontractivity estimates, and some
related large deviation bounds. In Section~\ref{Sec:Lemma} we present
the proof of Lemma~\ref{Lemma-biased-main} (which immediately implies
Lemma~\ref{Lemma-Main} as well). In Section~\ref{Sec:Theorem-Proof} we
show that Lemma~\ref{Lemma-biased-main} implies
Theorem~\ref{thm:biased-main} (and also Theorem~\ref{thm:main}). In
Section~\ref{Sec:Tightness} we discuss the tightness of our
results. Finally, in the Appendix we prove a decoupled version of
Lemma~\ref{Lemma-biased-main}.

\section{Preliminaries}
\label{Sec:Preliminaries}

\subsection{Biased Fourier-Walsh Expansion of Functions on the Discrete
Cube} \label{sec:sub:Fourier}

Throughout the paper we consider the discrete cube
$\Omega=\{0,1\}^n$, endowed with a probability product measure
$\mu=\mu_{p} \otimes \cdots \otimes \mu_{p}$, i.e.,
\[
\mu(x)= \mu \Big((x_1,\ldots,x_n)\Big)=\prod_{i=1}^n p^{x_i}
(1-p)^{1-x_i}.
\]
Elements of $\Omega$ are represented either by binary vectors of
length $n$, or by subsets of $\{1,2,\ldots,n\}$. Denote the set of
all real-valued functions on the discrete cube by $Y$. The inner
product of functions $f,g \in Y$ is defined as usual as
\[
\langle f,g \rangle = \expect{[fg]} = \sum_{x \in \{0,1\}^n} \mu(x)
f(x) g(x).
\]
This inner product induces a norm on $Y$:
\[
||f||_2 = \sqrt{\langle f,f \rangle} = \sqrt{\expect [f^2]}.
\]
\paragraph{Walsh Products.} Consider the functions $\{\ch_i\}_{i=1}^n$, defined as:
\[
\ch_i(x_1,\ldots,x_n) = \left\lbrace
  \begin{array}{c l}
    \sqrt{\frac{1-p}{p}}, & x_i=1\\
    -\sqrt{\frac{p}{1-p}}, & x_i=0.
  \end{array}
\right.
\]
As was observed in~\cite{Talagrand:94}, these functions constitute
an orthonormal system in $Y$ (with respect to the measure $\mu$).
Moreover, this system can be completed to an orthonormal basis in
$Y$ by defining
\[
\ch_T = \prod_{i \in T} \ch_i
\]
for all $T \subset \{1,\ldots,n\}$. The functions $\ch_T$ are called
(biased) Walsh products.

\medskip
\paragraph{Fourier-Walsh expansion.} Every function $f \in Y$ can
be represented by its Fourier-Walsh expansion with respect to the system
$\{\ch_T\}_{T \subset \{1,\ldots,n\}}$:
\[
f = \sum_{T \subset \{1,\ldots,n\}} \langle f,\ch_T \rangle \ch_T.
\]
The coefficients in this expansion are denoted
\[
\hat f(T) = \langle f,\ch_T \rangle.
\]
A coefficient $\hat f(T)$ is called {\it k-th level coefficient}
if $|T|=k$. By the Parseval identity, for all $f \in Y$ we have
\[
\sum_{T \subset \{1,\ldots,n\}} \hat f(T)^2 = ||f||_2^2.
\]

\medskip

\paragraph{Relation between Fourier-Walsh expansion, noise
stability, and influences} The noise stability of a Boolean
function can be expressed in a convenient way in terms of the
Fourier-Walsh expansion of the function.
\begin{claim}\label{Claim2.1}
For any function $f:\{0,1\}^n \rightarrow \{0,1\}$ and for any
$\epsilon>0$, we have
\[
\stab_{\epsilon}(f)= \sum_{S\neq\emptyset}
(1-\epsilon)^{|S|}\hat{f}(S)^2.
\]
\end{claim}
The assertion is obtained by direct computation in the case where
$f$ is a linear character, and it follows for general characters
by multiplicativity of expectation for independent random
variables. It then follows for the general case by linearity of
expectation.

\medskip

The influences are also related to the Fourier-Walsh
expansion. It can be easily shown that $p(1-p) \cdot \sum_{i=1}^n
I_i(f)=\sum_{S} |S| \hat f(S)^2$.
Moreover, the influences are specifically related to the
first-level Fourier coefficients. Indeed, denoting by
$x_{-i}\in\set{0,1}^{[n]\setminus\set{i}}$ the vector obtained
from $x$ by omitting the $i$'th coordinate, we have (for any
Boolean function $f$ and for any $1 \leq i \leq n$):
\[
|\hat f(\{i\})|= |\expect_x [\ch_i(x) f(x)]| \leq
\expect_{x_{-i}\in\set{0,1}^{[n]\setminus\set{i}}}\brac{
      \abs{\expect_{x_i\in\set{0,1}}\brac{\ch_i(x)f(x)}}} =
      \sqrt{p(1-p)} I_i(f),
\]
and thus,
\begin{equation}\label{Eq:Inf-Fourier1}
\sum_{i=1}^n \hat f(\{i\})^2 \leq \weight(f).
\end{equation}

These expressions of noise stability and influences in terms of
the Fourier-Walsh expansion play an important role in our proof.

\subsection{Sharp Bound on Large Deviations Using the Hypercontractive Inequality}
\label{sec:sub:Chernoff}

A crucial component in the proof of Lemma~\ref{Lemma-Main} is a bound
on the large deviations of low-degree multivariate
polynomials. Formally, for any $d \geq 1$, we would like to bound the
probability $\Pr[|f| \geq t]$, for every function $f$ whose Fourier
degree is at most $d$. In the uniform measure case, such bound was
obtained in~\cite{BKS} using the Bonami-Beckner hypercontractive
inequality~\cite{Bonami:70,Beckner:75}. In the biased case, one should
use a biased version of the Bonami-Beckner inequality instead, and the
strength of the obtained bound depends on the {\it hypercontractivity
  constant}, which depends on the bias. The optimal value of the
hypercontractivity constant for biased measures was obtained by
Oleszkiewicz in 2003~\cite{Oleszkiewicz:03}. For ease of
presentation, we cite a large deviation bound,  presented
in~\cite{DFKO}\footnote{In fact there is a small typo in the formula
  for $B(p)$ in~\cite{DFKO}, which is fixed here.}, that relies on a slightly weaker estimate of the
hypercontractivity constant.
\begin{definition}
For all $0<p<1$, let
\begin{equation}\label{Eq:B(p)}
B(p) = \frac{\frac{1-p}{p}-\frac{p}{1-p}}{2 \ln\frac{1-p}{p}} =
\frac{(1-p)-p}{2p(1-p) (\ln(1-p)-\ln p)}.
\end{equation}
\end{definition}
\begin{theorem}[Lemma 2.2 in~\cite{DFKO}]
\label{Thm:Oles}
 Let $f:\{0,1\}^n \rightarrow \mathbb{R}$ have Fourier degree at
most $d$, and assume that $||f||_2=1$. Then for any $t \geq
(2B(p)e)^{d/2}$,
\begin{equation}
  \label{eq:7}
\Pr[|f| \geq t] \leq \exp \Big(-\frac{d}{2B(p)e} t^{2/d} \Big).
\end{equation}
\end{theorem}

The next lemma, which easily follows from Fubini's theorem, allows
using large deviation bounds to evaluate certain expectations. The
integral that we get when we later apply it, using the bounds in
Theorem~\ref{Thm:Oles}, is considered in
Lemma~\ref{lemma:beckner-integral} .

\begin{lemma}\label{lemma:simple}
  Let $\Omega$ be a probability space, and let $f,g:\Omega\to\reals$ be
  functions, where $g$ is non-negative. For any real number $t$, let
  $L(t)\subseteq\Omega$ be defined by $L(t)=\set{x\;\colon\ g(x)>t}$,
  and let $\indicator_{L(t)}$ be the indicator of the set $L(t)$.
Then we have
\begin{displaymath}
  \expect_{x\in\Omega}  \brac{ f(x)g(x)} =
  \int_{t=0}^\infty \Parenth{\expect_{x\in\Omega} \brac{ f(x) \cdot \indicator_{L(t)}(x) }}\;dt
\end{displaymath}
\end{lemma}
\begin{proof}
  \begin{align*}
\expect_{x\in\Omega}  \brac{ f(x)g(x)} &=\expect_{x\in\Omega}  \brac{
  f(x)\cdot  \int_{t=0}^{g(x)}\indicator\;dt} = \expect_{x\in\Omega}  \brac{
  \int_{t=0}^{\infty} f(x)\cdot\indicator_{L(t)}(x)\;dt}\\&= \int_{t=0}^\infty
  \Parenth{\expect_{x\in\Omega} \brac{ f(x) \cdot \indicator_{L(t)}(x) }}\;dt,
  \end{align*}
where the last equality follows from Fubini's theorem.
\end{proof}

\begin{lemma}
  \label{lemma:beckner-integral}
Let $d \geq 2$ be a positive integer, and let $t_0$ be such that
$t_0>\Parenth{4B(p)\cdot e}^{(d-1)/2}$. Then
\begin{equation}
  \label{eq:8}
  \int_{t=t_0}^\infty t^2\cdot\exp\Parenth{-\frac{(d-1)}{2B(p)\cdot
      e}\cdot t^{2/(d-1)}}\;dt \leq 5B(p)\cdot e\cdot t_0^{3-\frac{2}{d-1}}\cdot\exp\Parenth{-\frac{(d-1)}{2B(p)\cdot
      e}\cdot t_0^{2/(d-1)}}.
\end{equation}
\end{lemma}
\begin{proof}
  To bound the l.h.s. of~(\ref{eq:8}) we first apply a change of
  variables, setting
  \begin{equation}
    \label{eq:10}
s=\frac{(d-1)}{2B(p)\cdot e}\cdot
  t^{2/(d-1)},
  \end{equation}
and obtaining
  \begin{equation}\label{eq:9}
\begin{split}
    \int_{t=t_0}^\infty t^2\cdot\exp\Parenth{-\frac{(d-1)}{2B(p)\cdot
        e}\cdot t^{2/(d-1)}}\;dt\qquad\qquad\qquad\qquad\qquad\qquad\qquad\qquad\qquad\qquad\\= \Parenth{\frac{2B(p)\cdot
        e}{d-1}}^{3(d-1)/2}\cdot \frac{d-1}2\cdot \int_{s=s_0}^\infty
    s^{(3d-5)/2}\cdot\exp(-s) \;ds\end{split}
\end{equation}
where $s_0=\frac{(d-1)}{2B(p)\cdot e}\cdot (t_0)^{2/(d-1)}$. If we
denote the integrand on the r.h.s.\ of (\ref{eq:9}) by $\varphi(s)$, one
notes that for $s\geq s_0$, $\varphi(s)$ is decreasing and
$\varphi(s+1)/\varphi(s)\leq \exp(-1/4)$ -- this follows from the
condition on $t_0$ and from (\ref{eq:10}). It therefore follows that
the integral on the r.h.s.\ of (\ref{eq:9}) is bounded by
$(s_0)^{(3m-5)/2}\cdot\exp(-s_0)\cdot\frac{1}{1-\exp(-1/4)} \leq 5
(s_0)^{(3m-5)/2}\cdot\exp(-s_0)$. Substituting into (\ref{eq:10})
gives the lemma.
\end{proof}

\section{Proof of Lemma~\ref{Lemma-biased-main}}
\label{Sec:Lemma}

\begin{notation}
Throughout the proof, we use a ``normalized'' variant of the
influences:
\[
I'_i(f)=\sqrt{p(1-p)} I_i(f).
\]
\end{notation}
\noindent This notation is only technical, and is intended to
avoid carrying the factor $\sqrt{p(1-p)}$ along the proof. Note
that $\weight(f)=\sum_{i=1}^n I'_i(f)^2$.

\subsection{Two key observations}

The key to the proof of Lemma~\ref{Lemma-biased-main} is based on two
observations, as was the proof in~\cite{Talagrand:97}.

\paragraph{First observation.} We write the space $\set{0,1}^n$ as a
product of two probability spaces $\set{0,1}^I$ and $\set{0,1}^J$. We
consider for every $j\in J$, the part of the Fourier-Walsh expansion
of $f$, which consists of Walsh products whose sole representative in
$J$ is $j$.

We now note that it is sufficient to prove that for every
partition $\{I,J\}$ of $\{1,\ldots,n\}$,
\begin{align}
\sum_{\stackrel{T \subset I}{ |T|=d-1}} \sum_{j \in J} \hat f(\{T,j\})^2
\leq   5 \cdot \Parenth{\frac{2B(p)\cdot
      e}{d-1}}^{d-1}\cdot\Parenth{\sum_{j\in J}{I'_j(f)}^2}\cdot \Parenth{\log\Parenth{\frac{1}{\sum_{j\in J}I'_j(f)^2}}}^{d-1}.
\label{newEq3.3.1}
\end{align}
The assertion of Lemma~\ref{Lemma-biased-main} will follow
from~\eqref{newEq3.3.1} by taking expectation over the partitions
$\set{I,J}$, such that every coordinate is independently put into
$J$ with probability $1/d$. We give the exact details at the end
of the proof.

\paragraph{Second observation.} The second observation is that we can
write the left-hand-side of~\eqref{newEq3.3.1} as the inner-product of
$f$ with a function of the form $\sum f_j'\ch_j$, where the
functions $\set{f_j'}$ are all of low degree, and depend only on
coordinates from $I$. The low degree of the $f_j'$'s will allow us to use
Theorem~\ref{Thm:Oles} to bound them.

For a given partition $\set{I,J}$
of $\set{1,\ldots,n}$ and an index $j\in J$, let
\[
f'_j=\sum_{T \subset I, |T|=m-1}\hat{f}(T,j)\ch_T.
\]
Note that $f'_j$ indeed depends only on coordinates from $I$.
We have

\begin{equation}
  \label{eq:6}
\inner{f'_j\cdot\ch_j,f} =\inner{\sum_{T
    \subset I, |T|=d-1}\hat{f}(T,j)\ch_{T\cup \{j\}},f} = \sum_{T \subset I,
  |T|=d-1} \hat f(\{T,j\})^2,
\end{equation}
and summing
over $j$ we have
\begin{equation}
  \label{eq:5}
\inner{ \sum_jf_j' \cdot \ch_j, f}= \sum_{T \subset I, |T|=d-1}
\sum_{j\in J}\hat f(\{T,j\})^2=\text{\ l.h.s. of~(\ref{newEq3.3.1})}.
\end{equation}
It will be convenient for us to normalize $f_j'$, hence we take
$f_j=f'_j/||f'_j||_2$. It follows from~(\ref{eq:6}) that for every $j\in J$,
\begin{equation}\label{eq:1}
\inner{ f_j \cdot \ch_j, f}= \Parenth{\sum_{T \subset I, |T|=d-1}  \hat f(\{T,j\})^2}^{1/2}.
\end{equation}

\subsection{Proof of~(\ref{newEq3.3.1}).}

Using~\eqref{eq:1} and the
fact that $f_j$ only depends on coordinates from $I$, we have for
every $j\in J$ that
\begin{align}
\sum_{T \subset I, |T|=d-1} \hat{f}(T,j)^2
&= \parenth{\inner{f_j\cdot \ch_j, f}}^2 = \parenth{\inner{f_j,\ch_j\cdot f}}^2 \label{eq:12}\\&=
\Parenth{\expect_{x\in \set{0,1}^I}\brac{f_j(x)\cdot\expect_{y\in
    \set{0,1}^J}\brac{\ch_j(x,y)f(x,y)}}}^2 \notag\\&\leq
\Parenth{\expect_{x\in \set{0,1}^I}\brac{|f_j(x)|\cdot\abs{\expect_{y\in
    \set{0,1}^J}\brac{\ch_j(x,y)f(x,y)}}}}^2.\label{eq:2}
\end{align}
We now use Lemma~\ref{lemma:simple} to bound (\ref{eq:2}) by
considering the two multiplicands in the expectations as functions
over $\set{0,1}^I$, and obtain
\begin{align*}
  (\ref{eq:2})&\leq \Parenth{ \int_{t=0}^\infty
    \expect_{x\in\set{0,1}^I}\brac{\indicator_{\set{|f_j(x)|>t}}\cdot
      \abs{\expect_{y\in
    \set{0,1}^J}\brac{\ch_j(x,y)f(x,y)}}}
 \;dt }^2 .
\end{align*}
Using the inequality $(a+b)^2\leq 2a^2+2b^2$ we thus have that for
any parameter $t_0$, (\ref{eq:2}) is bounded above by
\begin{align}
    &2\Parenth{ \int_{t=0}^{t_0}
    \expect_{x\in\set{0,1}^I}\brac{\indicator_{\set{|f_j(x)|>t}}\cdot
      \abs{\expect_{y\in
    \set{0,1}^J}\brac{\ch_j(x,y)f(x,y)}}} \;dt }^2
  \label{eq:3}\\ &+
  2 \Parenth{ \int_{t=t_0}^\infty
     \expect_{x\in\set{0,1}^I}\brac{\indicator_{\set{|f_j(x)|>t}}\cdot
       \abs{\expect_{y\in
     \set{0,1}^J}\brac{\ch_j(x,y)f(x,y)}}}
  \;dt }^2 .\label{eq:4}
\end{align}
We will bound separately each of these summands.

\paragraph{Bounding~(\ref{eq:3}).} For $z\in\set{0,1}^{[n]}$, we denote by $z_{-j}\in\set{0,1}^{[n]\setminus\set{j}}$ the vector
obtained from $z$ by omitting the $j$'th coordinate. Since an
indicator function is bounded by $1$, we have
\begin{align*}
  (\ref{eq:3}) &\leq
  2\Parenth{\int_{t=0}^{t_0}\expect_{x\in\set{0,1}^I}\brac{
      \abs{\expect_{y\in\set{0,1}^J}\brac{\ch_j(x,y)f(x,y)}}}}^2\\&\leq
  2t_0^2\cdot \Parenth{\expect_{x\in\set{0,1}^I}\brac{
      \abs{\expect_{y\in\set{0,1}^J}\brac{\ch_j(x,y)f(x,y)}}} }^2
  \\&\leq 2t_0^2\cdot \Parenth{\expect_{x\in\set{0,1}^{I}}\expect_{y'\in\set{0,1}^{J\setminus\set{j}}}\brac{
      \abs{\expect_{y_j\in\set{0,1}}\brac{\ch_j(x,y',y_j)f(x,y',y_j)}}}
  }^2 \\&=2t_0^2\cdot \Parenth{\expect_{z_{-j}\in\set{0,1}^{[n]\setminus\set{j}}}\brac{
      \abs{\expect_{z_j\in\set{0,1}}\brac{\ch_j(z)f(z)}}}
  }^2 =
  2t_0^2 \cdot{I'_j(f)}^2.
\end{align*}

\paragraph{Bounding (\ref{eq:4}).} In the computation below, we explain
some transitions below the corresponding line. We note that the
two last implications apply if $t_0>\parenth{4B(p)e}^{(d-1)/2}$,
which is indeed the case for the $t_0$ that is chosen later.
\begin{align*}
  (\ref{eq:4})&= 2 \Parenth{ \int_{t=t_0}^\infty
    \frac{1}{t}\cdot t\cdot
     \expect_{x\in\set{0,1}^I}\brac{\indicator_{\set{|f_j(x)|>t}}\cdot
       \abs{\expect_{y\in
     \set{0,1}^J}\brac{\ch_j(x,y)f(x,y)}}}
  \;dt }^2\\
&\leq 2 \Parenth{ \int_{t=t_0}^\infty
    \frac1{t^2}\;dt}  \cdot \int_{t=t_0}^\infty
    t^2\cdot \Parenth{
     \expect_{x\in\set{0,1}^I}\brac{\indicator_{\set{|f_j(x)|>t}}\cdot
       \abs{\expect_{y\in
     \set{0,1}^J}\brac{\ch_j(x,y)f(x,y)}}}
  }^2\;dt \\
\intertext{(by Cauchy-Schwarz)}&\leq \frac2{t_0}  \cdot \int_{t=t_0}^\infty
    t^2\cdot
    \Pr_{x\in\set{0,1}^I}\Brac{|f_j(x)|>t} \cdot \expect_{x\in\set{0,1}^I}\brac{
       \Parenth{\expect_{y\in
     \set{0,1}^J}\brac{\ch_j(x,y)f(x,y)}}^2
  }\;dt  \\
\intertext{(by applying Cauchy-Schwarz on the space $\set{0,1}^I$)}
&\leq \frac2{t_0}  \cdot  \expect_{x\in\set{0,1}^I}\brac{
       \Parenth{\expect_{y\in
     \set{0,1}^J}\brac{\ch_j(x,y)f(x,y)}}^2 }\cdot \int_{t=t_0}^\infty
    t^2\cdot
    \exp \Big(-\frac{(d-1)}{2B(p)e} t^{2/(d-1)} \Big)
  \;dt \\
\intertext{(by pulling the expectation outside of the integral, as it
  does not depend on $t$, and bounding the deviation of $f_j$ using
  Theorem~\ref{Thm:Oles} )}
&\leq
    10e \cdot t_0^{2-\frac{2}{d-1}}  \cdot B(p) \cdot\exp\Parenth{-\frac{(d-1)}{2B(p)\cdot
      e}\cdot t_0^{2/(d-1)}} \cdot
 \expect_{x\in\set{0,1}^I}\brac{
       \Parenth{\expect_{y\in
     \set{0,1}^J}\brac{\ch_j(x,y)f(x,y)}}^2 } \intertext{(using Lemma~\ref{lemma:beckner-integral}).}
\end{align*}

\paragraph{Proving inequality (\ref{newEq3.3.1}).} Since the sum of (\ref{eq:3}) and
(\ref{eq:4}) bounds the l.h.s. of (\ref{eq:12}), we have from the
above bounds that
\begin{align}
  \label{eq:11}
\begin{split}\sum_{T \subset I, |T|=d-1} \hat{f}(T,j)^2 \leq  \ \ 2t_0^2\cdot{I'_j(f)}^2+
\hskip 4in
\\
   + 10e \cdot t_0^{2-\frac{2}{d-1}}  \cdot B(p) \cdot\exp\Parenth{-\frac{(d-1)}{2B(p)\cdot
      e}\cdot t_0^{2/(d-1)}} \cdot
 \expect_{x\in\set{0,1}^I}\brac{
       \Parenth{\expect_{y\in
     \set{0,1}^J}\brac{\ch_j(x,y)f(x,y)}}^2 }.
\end{split}
\end{align}
We now use the following observation: for any fixed $x\in\set{0,1}^I$,
let $f_x:\set{0,1}^J\to \{0,1\}$ be defined by $f_x(y)=f(x,y)$. Then
\begin{align*}
\Parenth{\expect_{y\in
     \set{0,1}^J}\brac{\ch_j(x,y)f(x,y)}}^2=\widehat{f_x}(\set{j})^2.
\end{align*}
Since $||f_x||_2\leq 1$, by Parseval's identity we have that for every $x\in\set{0,1}^I$,
\begin{align}
  \label{eq:14}
\sum_{j\in J}\Parenth{\expect_{y\in
     \set{0,1}^J}\brac{\ch_j(x,y)f(x,y)}}^2\leq 1.
\end{align}
By summing~(\ref{eq:11}) over $j\in J$ and substituting (\ref{eq:14})
inside the expectation, we obtain that
\begin{align}
  \label{eq:15}
  \sum_{j\in J}\sum_{\stackrel{T \subset I}{ |T|=d-1}} \hat{f}(T,j)^2 \leq &
  2t_0^2\cdot \sum_{j\in J}{I'_j(f)}^2
   + 10e \cdot t_0^{2-\frac{2}{d-1}}  \cdot B(p) \cdot\exp\Parenth{-\frac{(d-1)}{2B(p)\cdot
      e}\cdot t_0^{2/(d-1)}}.
\end{align}
We now choose $t_0$ so that
\begin{align*}
  \exp\Parenth{-\frac{(d-1)}{2B(p)\cdot
      e}\cdot t_0^{2/(d-1)}} = \sum_{j\in J}{I'_j(f)}^2.
\end{align*}
A simple computation shows that
\begin{align}
  \label{eq:16}
  (t_0)^2 = \Parenth{\frac{2B(p)\cdot
      e}{d-1}}^{d-1}\cdot\Parenth{\log\Parenth{\frac{1}{\sum_{j\in J}I'_j(f)^2}}}^{d-1}
\end{align}
and by assumption~(\ref{Eq:Restriction3}), we have that
$(t_0)^{2/(d-1)}\geq 4B(p)\cdot e$, satisfying the requirement
mentioned in the bound on \eqref{eq:4}. We now substitute
$t_0$ in (\ref{eq:15}), obtaining the bound
\begin{align}
  \label{eq:17}
  \sum_{j\in J}\sum_{\stackrel{T \subset I}{ |T|=d-1}} \hat{f}(T,j)^2 \leq &
  5 \cdot \Parenth{\frac{2B(p)\cdot
      e}{d-1}}^{d-1}\cdot\Parenth{\sum_{j\in J}{I'_j(f)}^2}\cdot \Parenth{\log\Parenth{\frac{1}{\sum_{j\in J}I'_j(f)^2}}}^{d-1},
\end{align}
thus proving (\ref{newEq3.3.1}).

\paragraph{Completing the proof of Lemma~\ref{Lemma-biased-main}.} Let us choose
$J\subset\set{1,\ldots,n}$ to be a random subset, independently
containing each coordinate with probability $1/d$, and let $I=[n] \setminus J$.
For each subset $S\subseteq\{1,\ldots,n\}$ of
size $d$, the probability that it can be represented as a pair $(T,j)$
where $T\subseteq I$ and $j\in J$, in which case $\hat f(S)^2$ is included
as a summand in the left-hand-side of~\eqref{newEq3.3.1}, is
$((d-1)/d)^{d-1}>1/e$. Hence
\begin{align}
  \label{eq:18}
  \sum_{|S|=d} \hat{f}(S)^2
  \leq e\cdot \expect_{J} \brac{\sum_{j\in J}\sum_{\stackrel{T \subset I}{ |T|=d-1}} \hat{f}(T,j)^2}.
\end{align}
We wish to apply a similar argument to the r.h.s.\
of~\eqref{newEq3.3.1}, using the fact that each coordinate in
$\set{1,\ldots,n}$ appears in the sum with probability $1/d$. We
observe that the function $x\log(1/x)^{d-1}$ is concave in the
segment $[0,\exp(-(d-1))]$, and by assumption
(\ref{Eq:Restriction3}) the sum $\sum_{j\in J}{I'_j(f)}^2$ is in
this range for any $J\subseteq\set{1,\ldots,n}$. We thus have
\begin{align}
  \label{eq:13}
\begin{split}
  \expect_{J}&\brac{5 \cdot \Parenth{\frac{2B(p)\cdot
      e}{d-1}}^{d-1}\cdot\Parenth{\sum_{j\in
      J}{I'_j(f)}^2}\cdot \Parenth{\log\Parenth{\frac{1}{\sum_{j\in
          J}I'_j(f)^2}}}^{d-1}} \\ \leq &\frac 5d \cdot \Parenth{\frac{2B(p)\cdot
      e}{d-1}}^{d-1}\cdot\Parenth{\sum_{j\in
      \set{1,\ldots,n}}{I'_j(f)}^2}\cdot
      \Parenth{\log\Parenth{\frac{d}{\sum_{j\in\set{1,\ldots,n}}I'_j(f)^2}}}^{d-1}
      \\ =& \frac 5d \cdot \Parenth{\frac{2B(p)\cdot
      e}{d-1}}^{d-1}\cdot \weight(f) \cdot
      \Parenth{\log\Parenth{\frac{d}{\weight(f)}}}^{d-1}.
\end{split}
\end{align}
The combination of (\ref{eq:18}) and (\ref{eq:13}) completes the proof.

\section{Proof of Theorem~\ref{thm:biased-main}}
\label{Sec:Theorem-Proof}

Let $f:\set{0,1}^n\to \{0,1\}$ be a function, and let
$\weight=\weight(f)=p(1-p) \cdot \sum_{j=1}^n I_j(f)^2$. Our goal
is to show that
\begin{align}
  \label{eq:19}
\stab_\epsilon(f) \leq (6e+1)
\weight^{\alpha\cdot \epsilon},
\quad\text{where}\qquad \alpha=\frac{1}{\epsilon+\log(2B(p)e) +
  3\log\log(2B(p)e)}.
\end{align}

\noindent Recall that by Claim~\ref{Claim2.1}, we have
\begin{align}
  \label{eq:20}
\stab_{\epsilon}(f)= \sum_{S\neq\emptyset}
(1-\epsilon)^{|S|}\hat{f}(S)^2.
\end{align}

\noindent For some $L$ that we choose later, we write
\begin{align}
  \label{eq:22}
\stab_{\epsilon}(f)=\sum_{0<|S|\leq L}
(1-\epsilon)^{|S|}\hat{f}(S)^2 + \sum_{|S|>L}   (1-\epsilon)^{|S|}\hat{f}(S)^2,
\end{align}
and bound each of the terms separately.

\paragraph{Bounding the high degrees term.} The second term in
\eqref{eq:22} is dominated by the powers of $(1-\epsilon)$. Since
$||f||_2^2\leq 1$, we have from Parseval's identity that
\begin{align}
  \label{eq:23}
  \sum_{|S|> L}   (1-\epsilon)^{|S|}\hat{f}(S)^2 \leq
  (1-\epsilon)^{L+1} \cdot \sum_{|S|> L}   \hat{f}(S)^2 \leq (1-\epsilon)^{L+1}.
\end{align}

\paragraph{Bounding the low degrees term.} Here we neglect the powers
of $(1-\epsilon)$ and use Lemma~\ref{Lemma-biased-main} to bound the
Fourier coefficients of degree $d$ for each $1<d\leq L$ (for $d=1$
we use Equation~(\ref{Eq:Inf-Fourier1})).
\begin{align}
  \label{eq:24}
  \sum_{0<|S|\leq L}
(1-\epsilon)^{|S|}\hat{f}(S)^2 &\leq \sum_{d=1}^{L}\sum_{|S|=d}
\hat{f}(S)^2 \leq \weight +\sum_{d=2}^{L} \frac{5e}d \cdot \Parenth{\frac{2B(p)\cdot
      e}{d-1}}^{d-1}\cdot\weight\cdot \Parenth{\log\Parenth{\frac{d}{\weight}}}^{d-1}.
\end{align}
Assume that $L\leq B(p)\log(1/\weight)$. In that case for any $d$,
$1<d<L$, the ratio between the $(d+1)$ term and the $d$ term in
\eqref{eq:24} is bounded from below by $2$. Indeed, this ratio is
\begin{equation*}
\begin{split}
  \frac{d}{d+1} \cdot 2B(p)\cdot e \cdot \Parenth{\frac{d-1}{d}}^{d-1}
  \cdot\frac1d \cdot \Parenth{
    \frac{\log\Parenth{\nfrac{(d+1)}{\weight}}}
    {\log\Parenth{\nfrac{(d)}{\weight}}}}^{d-1} \cdot
  \log\Parenth{\frac{d+1}{\weight}} \\\geq \frac{2B(p)\cdot e}{d+1}
  \cdot \frac1e \cdot 1\cdot\log\Parenth{\frac{d+1}{\weight}} \geq
  \frac{2}{\log\Parenth{1/\weight}}\cdot
  \log\Parenth{\frac{d+1}{\weight}} \geq 2.
\end{split}
\end{equation*}
We can thus replace the sum in~\eqref{eq:24} by twice its last term,
thereby getting
\begin{align}\label{eq:25}
  \sum_{0<|S|\leq L}
(1-\epsilon)^{|S|}\hat{f}(S)^2 &\leq \weight + \frac{10e}L \cdot \Parenth{\frac{2B(p)\cdot
      e}{L-1}}^{L-1}\cdot\weight\cdot \Parenth{\log\Parenth{\frac{L}{\weight}}}^{L-1}.
\end{align}

\paragraph{Choosing the value of $L$.} We choose the value of $L$ such that the bound on
$\stab_\epsilon(f)$ obtained from (\ref{eq:22}), (\ref{eq:23}) and
(\ref{eq:25}) is minimized.  We thus take
\begin{align}\label{eq:26}
  L=\alpha\cdot\log(1/\weight),\ \ \text{ where
  }\quad\alpha=\frac1{\epsilon+\log(2B(p)e) + 3\log\log(2B(p)e)}.
\end{align}
Theorem~\ref{thm:biased-main} is obtained immediately from the following
claim.
\begin{claim}\label{sec:choosing-value-l}
  For $\alpha$ and $L$ as chosen in~(\ref{eq:26}), it holds that
\begin{align*}
(\ref{eq:23}) \leq  \weight^{\alpha\cdot\epsilon},
\end{align*}
and
\begin{align*}
  (\ref{eq:25})\leq  6e\cdot \weight^{\alpha\cdot\epsilon}.
\end{align*}
\end{claim}
\begin{proof}
The first inequality follows since
\begin{align*}
(\ref{eq:23}) \leq  (1-\epsilon)^{L}\leq \exp(-\epsilon\cdot
\alpha\cdot \log(1/\weight)) = \weight^{\alpha\cdot\epsilon}.
\end{align*}
To bound (\ref{eq:25}), we first note that
\begin{align*}
  (\ref{eq:25}) =& \weight+\frac{10e}L \cdot \Parenth{\frac{2B(p)\cdot
      e}\alpha}^{L-1}\cdot\Parenth{\frac{\alpha}{L}}^{L-1} \cdot \Parenth{\frac{L}{L-1}}^{L-1}
  \cdot\weight\cdot \Parenth{\log\Parenth{\frac{L}{\weight}}}^{L-1} \\
  \leq &\weight + \frac{10e^2}{L}
  \cdot \Parenth{\frac{1}{\log\Parenth{\nfrac1\weight}}}^{L-1}\cdot \Parenth{\log\Parenth{\frac{L}{\weight}}}^{L-1}
  \cdot \brac{\weight\cdot \Parenth{\frac{2B(p)\cdot
      e}\alpha}^{L-1}} \\ =
& \weight + \frac{10e^2}{L}
  \cdot \Parenth{1+\frac{\log\Parenth{L}}{\log\Parenth{\nfrac1\weight}}
    }^{L-1}
  \cdot \brac{\weight\cdot \Parenth{\frac{2B(p)\cdot
      e}\alpha}^{L-1}} \\
\leq
& \weight + \frac{10e^2}{L}
  \cdot \Parenth{1+\frac{\log\Parenth{L}}{\log\Parenth{\nfrac1\weight}}
    }^{\log(1/\weight)}
  \cdot \brac{\weight\cdot \Parenth{\frac{2B(p)\cdot
      e}\alpha}^{L-1}} \\
\leq &
\weight + 10e^2
  \cdot \brac{\weight\cdot \Parenth{\frac{2B(p)\cdot
      e}\alpha}^{L-1}},
\end{align*}
where the inequality before last holds because $\alpha\leq1$ (since
$B(p)\geq 1$ for all $p$).

\noindent Since $\weight\leq \weight^{\alpha\cdot\epsilon}$, to finish the claim
it remains to prove that
\begin{align}
  \label{eq:27}
  \weight\cdot \Parenth{\frac{2B(p)\cdot
      e}\alpha}^{L-1}\leq \frac1{2e}\cdot \weight^{\alpha\cdot\epsilon}.
\end{align}
Note that
\begin{align*}
  \weight\cdot \Parenth{\frac{2B(p)\cdot
      e}\alpha}^{L-1}\leq &\frac1{2B(p)\cdot e} \cdot  \weight\cdot \Parenth{\frac{2B(p)\cdot
      e}\alpha}^{\alpha\cdot\log(1/\weight)}\\
  = &\frac1{2B(p)\cdot e}\cdot \weight\cdot \Parenth{\weight}^{-\alpha\cdot\log\Parenth{\frac{2B(p)\cdot
      e}{\alpha}}} \\ \leq& \frac1{2 e}\cdot {\weight}^{1-\alpha\cdot\log\Parenth{\frac{2B(p)\cdot
      e}{\alpha}}}.
\end{align*}
Hence it is sufficient to show that the exponent above is higher than
$\alpha\cdot \epsilon$.  Since
\begin{align*}
  1-\alpha\cdot\log\Parenth{\frac{2B(p)\cdot
      e}{\alpha}} =& \frac{\epsilon+3\log\log(2B(p)\cdot
    e)-\log(1/\alpha)}{\epsilon+\log(2B(p)\cdot
    e)+3\log\log(2B(p)\cdot e)} \\ =& \alpha\cdot \epsilon +\frac{3\log\log(2B(p)\cdot
    e)-\log(1/\alpha)}{\epsilon+\log(2B(p)\cdot
    e)+3\log\log(2B(p)\cdot e)},
\end{align*}
we actually need to prove that
$3\log\log(2B(p)\cdot
    e)-\log(1/\alpha)\geq 0$. Substituting the value of $\alpha$ and
    simplifying, this reduces to
    \begin{align}\label{eq:28}
      \epsilon+\log(2B(p)\cdot e)+3\log\log (2B(p)\cdot e)\leq \Parenth{\log(2B(p)\cdot e)}^3.
    \end{align}
It is easy to verify that the function
$t^3-t-3\log t-\epsilon$ is monotone increasing in $t$ for every
$t\geq 1.5$, and for $\epsilon\leq 1$ its value for $t=1.5$ is
positive. Hence, since $\log
(2B(p)e)\geq 1.5$, (\ref{eq:28}) follows. This completes the proof of
Claim~\ref{sec:choosing-value-l}, and thus of Theorem~\ref{thm:biased-main}.
\end{proof}

\medskip

\noindent Theorem~\ref{thm:main} follows immediately from Theorem~\ref{thm:biased-main},
substituting in Equation~(\ref{eq:19}) $B(p)=1$ for $p=1/2$ and bounding
$\epsilon$ from above by $1$.

\section{Tightness of Lemma~\ref{Lemma-biased-main} and
  Theorem~\ref{thm:biased-main}}
\label{Sec:Tightness}

In this section we examine a variant of the ``tribes'' function
presented in~\cite{BenorLinial:90}. We show that the assertions of
Lemma~\ref{Lemma-biased-main} and Theorem~\ref{thm:biased-main}
are essentially tight for this function.

\medskip

\noindent The tribes function over $n$ coordinates with tribes of
size $r$ is defined as follows: we partition $n$ coordinates into
sets (tribes) of size $r$ each, and let the tribes function assume
the value $1$ if in at least one tribe all the coordinates are
equal to $1$, and $0$ otherwise. In order to make our function
approximately balanced with respect to the biased measure $\mu_p$,
we take $f$ to be a tribes function with tribes of size
\[
r=\frac{\log n - \log \log n + \log \log (1/p)}{\log(1/p)}.
\]
This variant of the tribes function was first considered
in~\cite{Talagrand:94}. It is easy to see that for this choice of
$r$ we have $\mathbb{E}_{x \sim \mu_p}[f(x)] \approx 1-1/e$. In
the following computations we use the symbol $\approx$ to denote
equality up to constant factors.

\subsection{Tightness of Lemma~\ref{Lemma-biased-main}}

\comment{Do we want to write something about the tightness for
$p=1/2$?}

To simplify the computations, we add a restriction on the range of
parameters we consider. We require that
\begin{equation}\label{Eq:Thesis-Rest1}
d \leq \min(1/p, \sqrt{r}, \log n / \log \log n).
\end{equation}
We note that a wide range of combinations of the parameters
satisfies this restriction. For example, if $p = 1/\log n$,
then~(\ref{Eq:Thesis-Rest1}) holds for all $d \leq \sqrt{\log n
/(2 \log \log n)}$.

\paragraph{Evaluating the right hand side of (\ref{eq:29}).}
For each $x \in \{0,1\}^n$, we have $f(x)\neq f(x \oplus e_i)$ if
and only if in the tribe of $i$, all the coordinates of $x$ other
than $x_i$ are ones, and in each of the other tribes, at least one
of the coordinates is zero. Thus, the influences of $f$ are:
\[
I_i(f)=\Pr[f(x)\neq f(x \oplus e_i)] = p^{r-1} (1-p^r)^{(n/r)-1}
\approx p^{r-1} = \frac{\log n}{p \log (1/p) \cdot n}.
\]
Summing over the values of $i$, we get
\begin{equation}
\weight(f)=p(1-p) \sum_{i=1}^n I_i(f)^2 \approx \frac{(\log
n)^2}{p (\log(1/p))^2 \cdot n}. \label{Eq6.1}
\end{equation}
Since by Assumption~(\ref{Eq:Thesis-Rest1}), $d \leq \frac{\log
n}{\log(1/p)}$, we have
\begin{equation}
\Parenth{\log \Parenth{\nfrac{1}{\weight(f)}}}^{d-1} \approx
\Parenth{(1-\frac{\log(1/p)}{\log n}) \log n}^{d-1} \approx (\log
n)^{d-1}. \label{Eq6.2}
\end{equation}
It follows that for our function $f$, the right hand side of
(\ref{eq:29}) is approximately
\begin{align}
\label{eq:30} \frac{5e}d \cdot \Parenth{\frac{2B(p)\cdot
e}{d-1}}^{d-1} \cdot \frac{(\log n)^{d+1}}{p (\log (1/p))^2 \cdot
n}.
\end{align}
Finally, since by Equation~(\ref{Eq:B(p)}), $B(p) \leq
\frac{1}{2p(1-p)\log((1-p)/p)}$, and since by
Assumption~(\ref{Eq:Thesis-Rest1}), $d \leq 1/p$, the right hand
side of (\ref{eq:29}) is at most
\begin{align}\label{eq:300}
\frac{5e}d \cdot \Parenth{\frac{e}{(d-1) \cdot
p(1-p)\log((1-p)/p)}}^{d-1} \cdot \frac{(\log n)^{d+1}}{p (\log
(1/p))^2 \cdot n} \approx \Parenth{\frac{e}{d}}^d \cdot
\frac{(\log n)^{d+1}}{p^{d} \cdot (\log (1/p))^{d+1} \cdot n}.
\end{align}

\paragraph{Evaluating the left hand side of (\ref{eq:29}).}
We compute a lower bound on the l.h.s. of (\ref{eq:29}) by
considering only part of the $d$-th level Fourier-Walsh
coefficients of $f$. We compute the coefficients of the form $\hat
f(\{i_1,\ldots,i_d\})$ where $\{i_1,\ldots,i_d\}$ belong to the
same tribe, the idea being that these are the dominant $d$-th
level coefficients.  We want to compute
\[
\hat f(\{i_1,\ldots,i_d\})= \expect[ f\cdot
\ch_{\{i_1,\ldots,i_d\}}].
\]
We divide $\{0,1\}^n$ into structures of $2^d$ values each
according to the coordinates $\{i_1,\ldots,i_d\}$. Note that a
structure does not contribute to $\expect [f\cdot
\ch_{\{i_1,\ldots,i_d\}}]$ if the value of $f$ on all the elements
of the structure is the same. Since $\{i_1,\ldots,i_d\}$ are all
in the same tribe, the only case in which $f$ is not constant on a
structure is when in the tribe containing $\{i_1,\ldots,i_d\}$,
all the other coordinates are ones, and in each of the other
tribes there is at least one zero element. For such structures,
$f$ assumes the value 1 only when all the coordinates
$\{i_1,\ldots,i_d\}$ are ones. Hence,
\begin{equation}
\expect[ f\cdot \ch_{\{i_1,\ldots,i_d\}}] =
\Parenth{-\sqrt{(1-p)/p}}^d p^r (1-p^r)^{(n/r)-1} \approx (-1)^d
p^{-d/2} p^r \approx (-1)^d p^{-d/2} \cdot \frac{\log n}{\log
(1/p) \cdot n}. \label{Eq6.4}
\end{equation}
The number of $d$-th level coefficients of this type is
$\frac{n}{r} {{r}\choose{d}}$. Since by
Assumption~(\ref{Eq:Thesis-Rest1}), we have $d \leq \sqrt{r}$ and
$d \leq \frac{\log \log n}{\log n}$, it follows that
\[
\frac{n}{r} {{r}\choose{d}} \geq \frac{n}{r} \frac{(r-d)^d}{d!}
\approx \frac{nr^{d-1}}{d!} \approx \frac{n (\log n)^{d-1}}{d!
(\log(1/p))^{d-1}}.
\]
Therefore, using Stirling's approximation $d! \approx \sqrt{2 \pi
d} (d/e)^d$, we get the following lower bound on the left hand
side of (\ref{eq:29}):
\begin{align}
\sum_{|S|=d} \hat f(S)^2 \gtrapprox& \frac{n (\log n)^{d-1}}{d!
(\log(1/p))^{d-1}} \cdot \frac{(\log n)^2}{p^d \cdot (\log
(1/p))^2 \cdot n^2} \approx \Parenth{\frac{e}{d}}^d \frac{(\log
n)^{d+1}}{\sqrt{d} \cdot p^d \cdot (\log(1/p))^{d+1} \cdot n}.
\label{Eq6.3}
\end{align}
Comparing expressions~(\ref{eq:300}) and~(\ref{Eq6.3}) shows that
the assertion of Lemma~\ref{Lemma-biased-main} is tight up to
factor $c\sqrt{d}$, which is small compared to the other terms in
the expressions in both sides of Equation~(\ref{eq:29}).

\subsection{Tightness of Theorem~\ref{thm:biased-main}}

\noindent In order to show that the assertion of
Theorem~\ref{thm:biased-main} is tight up to a constant factor in
the exponent, we have to prove that
\begin{equation}\label{Eq:Thesis-Tigh0}
\log \stab_\epsilon(f) \approx \alpha(\epsilon) \cdot \epsilon
\cdot \log \weight(f).
\end{equation}
To simplify the computations, we assume that $p \leq 1/\log n$, and in
particular we have $r \leq 1/p$ (we also deal separately with the
uniform measure case below).

\medskip

\noindent We compute a lower bound on $(\log \stab_\epsilon(f))$
by considering part of the $r$-th level Fourier-Walsh coefficients
of $f$. By Formula~(\ref{Eq6.4}), each of the coefficients $\hat
f(\{i_1,i_2,\ldots,i_r\})$ corresponding to a full tribe equals
\[
\Parenth{-\sqrt{(1-p)/p}}^r p^r (1-p^r)^{(n/r)-1} \approx (-1)^r
p^{-r/2} p^r =(-1)^r p^{r/2}.
\]
The number of coefficients of this form is $n/r$. Thus,
\[
\sum_{|S|=r} \hat f(S)^2 \gtrapprox \frac{np^r}{r} \approx \frac{n
\log n}{\log (1/p) \cdot n \cdot r} = \Theta(1).
\]
Hence, by Claim~\ref{Claim2.1},
\[
\stab_{\epsilon}(f) \geq \sum_{|S|=r} \hat f(S)^2 (1-\epsilon)^r
\gtrapprox (1-\epsilon)^r,
\]
and therefore,
\begin{equation}\label{Eq:Thesis-Tigh1}
\log(\stab_{\epsilon}(f)) \gtrapprox \frac{\log(1-\epsilon) \log
n} {\log(1/p)}.
\end{equation}
On the other hand, using Formula~(\ref{Eq6.2}) and the
approximation $\log (2B(p)e) \approx \log(1/p)$, we get
\begin{equation}\label{Eq:Thesis-Tigh2}
\alpha(\epsilon) \cdot \epsilon \cdot \log \weight(f) \approx
\frac{ -\epsilon \log n}{\log(1/p)}.
\end{equation}
Comparing Formulas~(\ref{Eq:Thesis-Tigh1})
and~(\ref{Eq:Thesis-Tigh2}) yields
Formula~(\ref{Eq:Thesis-Tigh0}). Moreover, it can be seen from the
proof that the exponent in the assertion of
Theorem~\ref{thm:biased-main} is tight up to factor
\[
(1+o(1))\frac{\epsilon}{-\log(1-\epsilon)},
\]
which tends to $1$ for small $\epsilon$.

\subsubsection{The Uniform Measure Case}
A similar computation shows the tightness up to constant factor in
the exponent in the case $p=1/2$ (namely, tightness of
Theorem~\ref{thm:main}). This time we compute a lower bound on
$(\log \stab_\epsilon(f))$ by considering the Fourier-Walsh
coefficients of $f$ all of whose coordinates are contained in the
same tribe. It is easy to check that for $p=1/2$, all such
coefficients are of order $\log_2 n /n$. For each of the $n/r$ tribes,
there are $2^r$ coefficients that correspond to its subsets, and thus,
\[
\sum_{\mbox{ S contained in a tribe }} \hat f(S)^2 \approx
\frac{n}{r} \cdot 2^r \cdot \frac{(\log_2 n)^2}{n^2} = \Theta(1).
\]
Hence, like in the previous case we get
\[
\stab_{\epsilon}(f) \geq \sum_{\mbox{ S contained in a tribe }}
\hat f(S)^2 (1-\epsilon)^r \gtrapprox (1-\epsilon)^r,
\]
and the rest of the proof is the same as in the previous case.

\medskip

We note that in this computation, the value of $C$ in the exponent tends
to $1$ as $\epsilon \rightarrow 0$. A refined computation can probably
improve the value of the constant, but we won't be able to match the value
$C=.234$ asserted in Theorem~\ref{thm:main}, since it follows from Corollary~12
in~\cite{MosselOdonnell-noise-sensitivity} that for the tribes function with $p=1/2$, we have
\[
\stab_{\epsilon}(f) \lessapprox \weight(f)^{\log_2 (e)/2} = \weight(f)^{.721}.
\]

\section{Acknowledgements}

We are grateful to Gil Kalai for encouraging us to work on this project, and
to Yuval Peres for useful suggestions.

\bibliographystyle{alpha}
\bibliography{bib1}

\section{Appendix}
\label{sec:appendix}
\comment{Remove the "section" sign here}

In this appendix we prove a decoupled variant of
Lemma~\ref{Lemma-Main}, stated as Theorem~\ref{Thm:Decoupled}
below. This variant generalizes Theorem~2.4 of~\cite{Talagrand:96},
that was used by Talagrand to establish a lower bound on the
correlation between monotone subsets of the discrete cube. While
employing the same basic idea, our proof is shorter than Talagrand's
proof, and applies also to a biased measure $\mu_p$ on the discrete
cube. For the sake of generality, we present the proof in the biased
case, thus providing a decoupled variant of
Lemma~\ref{Lemma-biased-main}. As in the proof of
Lemma~\ref{Lemma-biased-main}, we use throughout the appendix the
normalized variant of the influences: $I'_i(f)=\sqrt{p(1-p)} I_i(f)$.


\bigskip

\noindent We start with a  generalization of Lemma~3.1
in~\cite{Talagrand:96}. The proof is a slight modification of the
proof of Lemma~\ref{Lemma-biased-main}.
\begin{lemma}
Let $f:\{0,1\}^n \rightarrow \{0,1\}$ be a function, and let
$\{I,J\}$ be a partition of $\{1,\ldots,n\}$. For $t>0$, denote
\[
L^f_t=\{j \in J: \sum_{\stackrel{T \subset I}{ |T|=m-1}} \hat f
(T,j)^2
> t \cdot I'_j(f)^2\}.
\]
Then for all $t \geq 4 \cdot \Parenth{4B(p)\cdot e}^{d-1}$,
\begin{equation}
\sum_{j \in L^f_t} I'_j(f)^2 \leq 5e \cdot (t/4)^{-\frac{1}{d-1}}
\cdot B(p) \cdot\exp\Parenth{-\frac{(d-1)}{2B(p)\cdot
      e}\cdot (t/4)^{1/(d-1)}}.
\end{equation}
\label{Lemma:First}
\end{lemma}
\begin{proof}
Exactly the same argument as used in the proof of
Lemma~\ref{Lemma-biased-main} (see Equation~\ref{eq:15}) shows
that for all $t_0 \geq \Parenth{4B(p)\cdot e}^{(d-1)/2}$,
\begin{align}
  \label{eq:Appendix1}
  \sum_{j\in L^f_t}\sum_{\stackrel{T \subset I}{ |T|=d-1}} \hat{f}(T,j)^2 \leq &
  2t_0^2\cdot \sum_{j\in L^f_t}{I'_j(f)}^2
   + 10e \cdot t_0^{2-\frac{2}{d-1}}  \cdot B(p) \cdot\exp\Parenth{-\frac{(d-1)}{2B(p)\cdot
      e}\cdot t_0^{2/(d-1)}}.
\end{align}
On the other hand, by the definition of $L^f_t$ we have
\begin{align}
\label{eq:Appendix2} \sum_{j\in L^f_t}\sum_{\stackrel{T \subset
I}{ |T|=d-1}} \hat{f}(T,j)^2 \geq \sum_{j\in L^f_t} t \cdot
I'_j(f)^2 = t \cdot \sum_{j\in L^f_t} I'_j(f)^2.
\end{align}
Taking $t_0=\sqrt{t}/2$ and combining
Inequalities~(\ref{eq:Appendix1}) and~(\ref{eq:Appendix2}), we get
\begin{align}
t \cdot \sum_{j\in L^f_t} I'_j(f)^2 \leq (t/2) \cdot \sum_{j\in
L^f_t}{I'_j(f)}^2
   + 10e \cdot (t/4)^{1-\frac{1}{d-1}}  \cdot B(p) \cdot\exp\Parenth{-\frac{(d-1)}{2B(p)\cdot
      e}\cdot (t/4)^{1/(d-1)}},
\end{align}
and simplification yields the assertion.
\end{proof}

\medskip

\noindent We now present a decoupled variant of
Lemma~\ref{Lemma:First}. The proof is a series of applications of
the Cauchy-Schwarz inequality.
\begin{lemma}
\label{Lemma:First-Decoupled} Let $f,g:\{0,1\}^n \rightarrow
\{0,1\}$ be functions, such that
\begin{equation}\label{eq:Appendix3}
\sum_{j=1}^n I'_j(f)^2 \leq 1, \qquad \mbox{ and } \qquad
\sum_{j=1}^n I'_j(g)^2 \leq 1,
\end{equation}
and let $\{I,J\}$ be a partition of $\{1,\ldots,n\}$. For $t>0$,
denote
\[
L_t=\{j \in J: \sum_{\stackrel{T \subset I}{ |T|=d-1}} \hat f
(T,j) \hat g (T,j) > t \cdot I'_j(f) I'_j(g)\}.
\]
Then for all $t \geq 4 \cdot \Parenth{4B(p)\cdot e}^{d-1}$,
\begin{equation}\label{eq:Appendix4}
\sum_{j \in L_t} I'_j(f) I'_j(g) \leq 2 \cdot (5e)^{1/2} \cdot
(t/4)^{-\frac{1}{2(d-1)}}  \cdot (B(p))^{1/2}
\cdot\exp\Parenth{-\frac{(d-1)}{4B(p)\cdot
      e}\cdot (t/4)^{1/(d-1)}}.
\end{equation}
\end{lemma}
\begin{proof}
Define the sets $L^f_t$ and $L^g_t$ as in Lemma~\ref{Lemma:First}
above.
By the Cauchy-Schwarz inequality and
Assumption~(\ref{eq:Appendix3}),
\[
\sum_{j \in L^f_t} I'_j(f) I'_j(g) \leq \Parenth{\sum_{j \in
L^f_t} I'_j(f)^2}^{1/2} \Parenth{\sum_{j \in L^f_t}
I'_j(g)^2}^{1/2} \leq \Parenth{\sum_{j \in L^f_t}
I'_j(f)^2}^{1/2},
\]
and thus by Lemma~\ref{Lemma:First},
\[
\sum_{j \in L^f_t} I'_j(f) I'_j(g) \leq (5e)^{1/2} \cdot
(t/4)^{-\frac{1}{2(d-1)}}  \cdot (B(p))^{1/2}
\cdot\exp\Parenth{-\frac{(d-1)}{4B(p)\cdot
      e}\cdot (t/4)^{1/(d-1)}}.
\]
The same holds for $\sum_{j \in L^g_t} I'_j(f) I'_j(g)$. Now we
note that by the Cauchy-Schwarz inequality, $L_t \subseteq L^f_t
\cup L^g_t$, and thus,
\begin{align*}
\sum_{j \in L_t} I'_j(f) I'_j(g) &\leq \sum_{j \in L^f_t} I'_j(f)
I'_j(g) +
\sum_{j \in L^g_t} I'_j(f) I'_j(g) \\
&\leq 2 \cdot (5e)^{1/2} \cdot (t/4)^{-\frac{1}{2(d-1)}}  \cdot
(B(p))^{1/2} \cdot\exp\Parenth{-\frac{(d-1)}{4B(p)\cdot
      e}\cdot (t/4)^{1/(d-1)}},
\end{align*}
as asserted.
\end{proof}

\medskip

\noindent Now we are ready to present the decoupled version of
Lemma~\ref{Lemma-biased-main}.
\begin{theorem}\label{Thm:Decoupled}
For all $d \geq 2$, and for any functions $f,g:\{0,1\}^n
\rightarrow \{0,1\}$ such that:
\begin{equation}
\sum_{j=1}^n I'_j(f)^2 \leq 1, \quad \qquad \sum_{j=1}^n I'_j(g)^2
\leq 1, \qquad \mbox{and} \qquad \sum_{j=1}^n I'_j(f)I'_j(g) \leq
\exp(-2(d-1)),\qquad \label{eq:Appendix5}
\end{equation}
we have
\begin{align}\label{eq:Appendix6}
\sum_{|S|=d} \hat f(S) \hat g(S) \leq \frac{70e}d \cdot
\Parenth{\frac{4B(p)\cdot
      e}{d-1}}^{d-1}\cdot\Parenth{\sum_{j=1}^n I'_j(f) I'_j(g) }\cdot
      \Parenth{\log\Parenth{\frac{d}{{\sum}_{\substack{j\leq
      n}}I'_j(f)I'_j(g)}}}^{d-1}.
\end{align}
\end{theorem}
\begin{proof}
As in the proof of Lemma~\ref{Lemma-biased-main}, we first
consider a partition $\{I,J\}$ of $\{1,\ldots,n\}$, and prove that
\begin{align}\label{eq:Appendix7}
  \begin{split}
\sum_{j \in J} \sum_{\stackrel{T \subset I}{ |T|=d-1}} \hat f(T,j)
\hat g(T,j) \leq 70 \cdot \Parenth{\frac{4B(p)\cdot
      e}{d-1}}^{d-1}\cdot\Parenth{\sum_{j\in J} I'_j(f) I'_j(g) }\cdot & \\
      \cdot
      \Parenth{\log\Parenth{\frac{1}{\sum_{j\in J}I'_j(f)I'_j(g)}}}^{d-1}
      &.
  \end{split}
\end{align}
We apply Lemma~\ref{lemma:simple} with $\Omega=J$ endowed with the
uniform measure, and the functions
\[
f_1(j) = I'_j(f) I'_j(g), \qquad \mbox{ and } \qquad f_2(j) =
\frac{\sum_{T \subset I, |T|=d-1} \hat f (T,j) \hat g
(T,j)}{I'_j(f) I'_j(g)}.
\]
Noting that the set $L(t)=\{j:f_2(j) > t\}$ is exactly the set
$L_t$ defined in Lemma~\ref{Lemma:First-Decoupled}, we get
\begin{align}\label{eq:Appendix8}
\sum_{j\in J}\sum_{\stackrel{T \subset I}{ |T|=d-1}} \hat{f}(T,j)
\hat{g}(T,j) = \int_{t=0}^{\infty} \Parenth{\sum_{j \in L_t}
I'_j(f)I'_j(g)} dt.
\end{align}
Therefore, using Lemma~\ref{Lemma:First-Decoupled}, we have, for
all $t_1 \geq 4 \cdot \Parenth{4B(p)\cdot e}^{d-1}$,
\begin{align*}
(\ref{eq:Appendix8}) = \int_{t=0}^{t_1} \Parenth{\sum_{j \in L_t}
I'_j(f)I'_j(g)} dt + \int_{t=t_1}^{\infty} \Parenth{\sum_{j \in
L_t} I'_j(f)I'_j(g)} dt \leq t_1 \sum_{j \in J} I'_j(f)I'_j(g) +
\\
+ \int_{t=t_1}^{\infty} \Parenth{2 \cdot (5e)^{1/2} \cdot
(t/4)^{-\frac{1}{2(d-1)}}  \cdot (B(p))^{1/2}
\cdot\exp\Parenth{-\frac{(d-1)}{4B(p)\cdot
      e}\cdot (t/4)^{1/(d-1)}}} dt.
\end{align*}

\noindent We evaluate the integral in the same way as the integral
in Lemma~\ref{lemma:beckner-integral}, and obtain that for all
$t_1 \geq 4 \cdot (8B(p) \cdot e)^{d-1}$,
\begin{align}\label{eq:Appendix9}
(\ref{eq:Appendix8}) \leq t_1 \sum_{j \in J} I'_j(f)I'_j(g) +
18.74 \cdot (B(p))^{1/2} \cdot (16B(p)e) \cdot
(t_1/4)^{1-\frac{3}{2d-2}} \cdot \exp
\Parenth{-\frac{(d-1)}{4B(p)\cdot e}\cdot (t_1/4)^{1/(d-1)}}.
\end{align}
We then choose $t_1$ such that
\[
\exp \Parenth{-\frac{(d-1)}{4B(p)\cdot e}\cdot (t_1/4)^{1/(d-1)}}
= \sum_{j \in J} I'_j(f)I'_j(g).
\]
(Note that due to Assumption~(\ref{eq:Appendix5}), we have $t_1
\geq 4 \cdot (8B(p) \cdot e)^{d-1}$, as required). Substituting
$t_1$ into Inequality~(\ref{eq:Appendix9}), we get
\begin{align}
(\ref{eq:Appendix8}) \leq 70 \cdot
\Parenth{\frac{4B(p)e}{d-1}}^{d-1} \Parenth{\sum_{j \in J}
I'_j(f)I'_j(g)} \Parenth{\log \frac{1}{\sum_{j \in J}
I'_j(f)I'_j(g)}}^{d-1},
\end{align}
proving~(\ref{eq:Appendix7}). Finally, the derivation
of~(\ref{eq:Appendix6}) from~(\ref{eq:Appendix7}) is exactly the
same as the last step in the proof of
Lemma~\ref{Lemma-biased-main}.
\end{proof}

\end{document}